\documentclass[12pt]{amsart}

\usepackage{amsfonts,amsmath,latexsym,amssymb,verbatim,amsbsy,amsthm}

\usepackage{graphicx}

\usepackage{array}
\usepackage[labelsep=quad,indention=10pt]{caption}
\usepackage[list=true]{subcaption}
\usepackage{multirow}

\theoremstyle{plain}
\newtheorem{THEOREM}{Theorem}[section]

\newtheorem{corollary}[THEOREM]{Corollary}
\newtheorem{lemma}[THEOREM]{Lemma}
\newtheorem{proposition}[THEOREM]{Proposition}

\theoremstyle{definition}

\newtheorem{definition}[THEOREM]{Definition}

\theoremstyle{remark}

\newcommand{\lem}[1]{Lemma~\ref{#1}}

\newcommand{\prop}[1]{Proposition~\ref{#1}}

\def \a {\alpha}
\def \b {\beta}
\def \d {\delta}
\def \z {\zeta}

\def \g {\gamma}

\def \e {\varepsilon}
\def \f {\varphi}

\def \l {\lambda}

\def \n {\nabla}
\def \s {\sigma}

\def \th {\theta}

\def \w {\omega}

\def \bs {\boldsymbol{\sigma}}

\def \bu {{\bf u}}
\def \bv {{\bf v}}
\def \btau {\boldsymbol{\tau}}
\def \bnu {\boldsymbol{ \nu}}

\def \cD {\mathcal{D}}

\newcommand{\N}{\ensuremath{\mathbb{N}}}   
\newcommand{\R}{\ensuremath{\mathbb{R}}}   
\newcommand{\T}{\ensuremath{\mathbb{T}}}   

\def \lan {\langle}
\def \ran {\rangle}
\def \p {\partial}
\def \ra {\rightarrow}
\def \ss {\subset}
\def \bs {\backslash}

\newcommand{\rest}[2]{#1\raisebox{-0.3ex}{\mbox{$\mid_{#2}$}}}

\DeclareMathOperator{\loc}{loc} %
\DeclareMathOperator{\sign}{sgn} %
\begin{document}

\title{2D homogeneous solutions to the Euler equation}
\author{Xue Luo}
\address{School of Mathematics and Systems Science, Beihang University, 37 Xueyuan Road, Haidian District, Beijing, P. R. China. 100191}
\email{xluo@buaa.edu.cn}

\author{Roman Shvydkoy}
\address{Department of Mathematics, Statistics, and Computer Science
University of Illinois at Chicago
322 Science and Engineering Offices (M/C 249)
851 S. Morgan Street
Chicago, IL 60607-7045}
\email{shvydkoy@uic.edu}

\thanks{X.L. acknowledges the support of the Grant No. YWF-14-RSC=026 from Beihang University. The work of R.S. is partially supported by NSF grant DMS-1210896. A part of this project was completed during R.S.'s visit at the Institute for Mathematics and its Applications, Minnesota. He thanks the Institute for hospitality and Vladimir Sverak for stimulating conversations.}

\date{\today}

\begin{abstract} In this paper we study classification of homogeneous solutions to the stationary Euler equation with locally finite energy. Written in the form  $\bu = \n^\perp \Psi$, $\Psi(r,\th) = r^{\l} \psi(\th)$, for $\l >0$, we show that only trivial solutions exist in the range $0<\l<1/2$, i.e. parallel shear and rotational flows. In other cases many new solutions are exhibited that have hyperbolic, parabolic and elliptic structure of streamlines. In particular, for $\l>9/2$ the number of different non-trivial elliptic solutions is equal to the cardinality of the set $(2,\sqrt{2\l}) \cap \N$. The case $\l = 2/3$ is relevant to Onsager's conjecture. We underline the reasons why no anomalous dissipation of energy occurs for such solutions despite their critical Besov regularity $1/3$.
\end{abstract}

\maketitle


\section{Introduction} The Euler equations of motion of an ideal incompressible fluid are given by
\begin{equation}\label{see}
\begin{split}
\p_t \bu+ \bu \cdot \n \bu + \n p & = 0 \\
 \n \cdot \bu & = 0.
 \end{split}
 \end{equation}
The system in the open space has a two parameter family of scaling symmetries: 
\[
(\bu,p) \ra (\l^\a \bu(\l^{\a+\b}t, \l^\b x), \l^{2\a} p(\l^{\a+\b}t, \l^\b x))
\]
which allow to look for scaling invariant, or self-similar, solutions of the form
\begin{equation}\label{ansatz}
\begin{split}
\bu(x,t) & = \frac{1}{(T - t)^{\frac{\a}{1+\a}}} \bv  \left( \frac{x-x^*}{(T - t)^{\frac{1}{1+\a}}}\right) \\
p(x,t) & =  \frac{1}{(T - t)^{\frac{2\a}{1+\a}}} q\left( \frac{x-x^*}{(T - t)^{\frac{1}{1+\a}}}\right) + c(t),
\end{split}
\end{equation}
where $\a >-1$ and $\bv$, $q$ are profiles of velocity and pressure, respectively, in self-similar variables. Recently, such solutions were excluded  under various decay assumptions on $\bv$ at infinity, provided $\bv$ and $q$ are locally smooth functions, see \cite{ChShv,BrShv} and references therein. An example of such solution would demonstrate blow-up which is a major open question in 3D. Examples of non-locally smooth solutions in the form of \eqref{ansatz} are abundant, but all of them are stationary, hence homogeneous, e.g. $\bv(y) = \frac{y^\perp}{|y|^{\a+1}}$ in 2D. 
In this paper we make an attempt to give a systematic classification of such two-dimensional homogeneous solutions. The question has been previously raised in a just few sources  -- in relation to Lie-symmetries in \cite{Olver} and within a general collection of special solutions in \cite{PZ}. However, no description was provided for those. More recently, homogeneous solutions of degree $-1/3$ were highlighted as candidates for possessing anomalous energy dissipation as such solutions fall precisely into the so-called Onsager critical local class $B^{1/3}_{3,\infty}$.  It was shown in \cite{shv-lectures} that the energy flux in fact vanishes for such solutions, however the underlying reason for that was missing. These findings, motivated us to look into the question of classification in great detail. 

We restrict ourselves to solutions with locally finite energy only. We write
\begin{equation}\label{homo}
\bu(r,\th) = r^q [ u_1(\th) \btau + u_2(\th) \bnu ], \quad q > -1.
\end{equation}
Here $(r,\th)$ are polar coordinates, $\btau = \lan -\sin \th, \cos \th \ran$ and $\bnu = \lan \cos \th, \sin \th \ran$ are the vectors of the standard local basis.  
The case $q=-1$ is important as well for it includes the classical point vortex. However in that case we can show that all such solutions are of the form
\[
\bu = \frac{1}{r} (A \btau + B \bnu), \quad A,B \in \R.
\]
So, we  only focus on the case $q> - 1$. In that case one can find a global stream-function $\bu = \n^\perp \Psi $, where $\Psi(r,\th) = r^{\l} \psi(\th)$ with $\psi(\th) = u_1/ \l$, $\l = q+1$. Furthermore, the pressure $p(r,\th) = r^{2q}P(\th)$ has no dependence on $\th$: $P(\th) = P$. Plugging $\bu$ and $p$ into the Euler system one reads off a second order ODE (see \eqref{ODE}) for $\psi$. We thus are looking for $2\pi$-periodic solutions with $\psi \in H^1(\T)$. The ODE has a Hamiltonian structure with pressure $P$ being the Hamiltonian. Solutions to nonlinear Hamiltonian systems and the issues related to the periods of solutions is a classical subject, see \cite{chicone,cima} and references therein. They cannot always be found explicitly however in our case their streamline geometry can be described qualitatively. To get a glimpse on the typical  structure of solutions, let us consider the case $\l =2$. It turns out in this case we can characterize all solutions to be of the form (up to symmetries, see below): $ \psi = \g_1 + \g_2 \cos(2\th)$, $P = 2(\g_1^2 - \g_2^2)$, $\Psi = (\g_1+\g_2)x^2+ (\g_1-\g_2)y^2$. We can see that elliptic, parallel shear, and hyperbolic configuration of streamlines are determined by $P>0$, $P=0$, $P<0$, respectively. Another example is given by $\l = \frac12$,  $\psi = \sqrt{\g_1 + \g_2 \cos(\th)}$, $P =-\frac{\g_1}{4}$, $\Psi = \sqrt{\g_1 r + \g_2 x}$. Here transition from elliptic, $|\g_2| < \g_1$, to hyperbolic, $|\g_2| > \g_1$, cases is given by a truly parabolic solution, when $|\g_2| = \g_1$, and it is not determined by the sign of $P$, rather by the sign of another conserved quantity, the Bernoulli function $B = (2P+\l^2 \psi^2+(\psi')^2)\psi^{\frac{2}{\l}-2}$. Note that here the hyperbolic solutions are not in $H^1$, which along with the requirement of  life-period $2\pi$, is another major obstacle for existence of solutions for various values of $\l$, $P$, and $B$. In Section~\ref{s:summary} we give detailed summary of all cases where we can construct such solutions, which we organize into separate tables  for each type of solutions. Here we give a few highlights of the obtained results.

First, there is no trivial solutions for $0<\l<\frac12$. By a trivial solution we understand flows that exist regardless of $\l$: rotational flow $\psi = const$, corresponding to extreme points of the pressure Hamiltonian $P$, and paraller shear flow $\psi = | \cos(\th)|^\l$, corresponding to $P=0$, a separatrix between elliptic and hyperbolic regions. Cases $\l = \frac12, 2$ are completely classified by the examples above. For $\frac12 <\l <1$ infinitely many hyperbolic solutions can be constructed for $P>0$, but no such solutions exist for $P<0$. There are no non-trivial elliptic solutions for $\frac12 <\l \leq \frac34$. Elliptic case in the range $\frac34< \l <1$ remains uncertain, and we will comment on it later. For $\l=1$ all solutions are parallel shear flows. The range $\l>1$ is in a sense conjugate to $0<\l<1$. In that range infinitely many hyperbolic solutions can be constructed for $P<0$, but no such solutions exist for $P>0$. For elliptic solutions the range $1<\l<\frac43$ is uncertain too, and no such solutions exist for $\l \in [4/3, 2) \cup (2, 9/2]$ (the case $\l=2$ being exceptional, see example above). Interestingly, non-trivial elliptic solutions emerge as $\l$ crosses beyond $9/2$. Their number is equal precisely to the number of integers in the interval $(2,\sqrt{2\l})$. 

The paper is organized as follows. In Section~\ref{s:general} we discuss  structure of the ODE for $\psi$ to satisfy, its weak formulation and relation to the Euler equation. We give a general description of solutions, where the main observation is that (local) solutions that vanish at two points can be glued together to form new solutions and that new solutions belong to $H^1$ automatically as long as each piece does. We also give full classifications for $\l = 0, \frac12, 1, 2$. We explain the Hamiltonian structure of the ODE and conjugacy of cases $0<\l<1$ and $1<\l$. Then in Section~\ref{s:desc} we focus on $\l>1$, and examine the period function $T$ for solutions of the system. We employ the results of Chicone \cite{chicone} and Cima, et al \cite{cima} to prove monotonicity of $T$ and finding the ranges of $T$ in the elliptic regions. The results for $0<\l<1$ follow by conjugacy, and in Section~\ref{s:summary} we collect the classification tables together. We also come back to the case $\l = 2/3$ and elaborate on its relation to Onsager's conjecture. 

In the elliptic case of $\l \in (3/4,1) \cup (1,4/3)$ the above method of Chicone et al fails, i.e. the preconditions for monotonicity are not satisfied in our case. However based on our numerical evidence we believe that the periods $T$ are still monotone, and thus elliptic solutions do not exist in that range either.

\section{General considerations and some special cases}\label{s:general}

\subsection{Euler equation in polar coordinates}
We consider homogeneous solutions to the Euler equation of the form
\begin{equation}\label{homo}
\bu(r,\th) = r^q [ u_1(\th) \btau + u_2(\th) \bnu ], \quad q \geq -1.
\end{equation}
Here $(r,\th)$ are polar coordinates, $\btau = \lan -\sin \th, \cos \th \ran$ and $\bnu = \lan \cos \th, \sin \th \ran$ are the vectors of the standard local basis.  Notice the formulas: $\p_\th \btau = -\bnu$, $\p_\th \bnu = \btau$. The divergence in polar coordinates is given by $\n \cdot \bu = \p_r \bu \cdot \bnu + \frac{1}{r} \p_\th \bu \cdot \btau$. Thus, the incompressibility condition takes the form
\begin{equation}\label{div-free}
(q+1) u_2 + u_1'  = 0.
\end{equation}
To write the Euler equation in polar notation, first we write the Jacobi matrix of $\bu$ as $\n \bu = \p_r \bu \otimes \bnu + \frac{1}{r} \p_\th \bu \otimes \btau$. In view of \eqref{homo}, \eqref{div-free} we obtain
 \[
 \n \bu = r^{q-1} \left[ - q u_2 \btau \otimes \btau + q u_1 \btau \otimes \bnu + (u_2'-u_1) \bnu \otimes \btau + q u_2 \bnu \otimes \bnu \right].
 \]
 The nonlinear term becomes 
 \[
 \bu \cdot \n \bu = r^{2q-1} ( u_1u_2' - u_1^2+q u_2^2 ) \bnu.
 \]
 Since there is no tangential part, the natural pressure ansatz must be $p = r^{2q} P$, there $P$ is constant. We thus obtain
 \begin{equation}\label{ee-polar}
 u_1u_2' - u_1^2+q u_2^2 + 2q P = 0.
\end{equation}
 
\subsection{Case $q  = -1$}
This case corresponds to the classical point vortex solutions, and in fact it is easy to find all other $(-1)$-homogeneous solutions. The incompressibility condition \eqref{div-free} forces $u_1$ to be constant. Then \eqref{ee-polar} reads
\[
u_1 u_2' = u_2^2 + u_1^2 + 2P.
\]
If $u_1 = 0$, then $u_2$ is constant. If $u_1 \neq 0$, then the above Riccati equation on $u_2$ has no smooth $2\pi$-periodic solutions, except for the trivial one $u_2 = const$. 
We conclude that all $(-1)$-homogeneous solutions are of the form
\[
\bu = \frac{1}{r} (A \btau + B \bnu), \quad A,B \in \R.
\]

\subsection{Stream-function} From now on we will be concerned with the case $q > -1$ (note that $q<-1$ yields solution with locally infinite energy, thus of lesser interest). We will find that using $\l = q+1$ as a homogeneity parameter, which corresponds to homogeneity of the steam-function if it exists, greatly simplifies the notation. Thus, under the standing assumption, $\l > 0$. Despite the fact that $\R^2\bs \{0\}$ is not simply connected, all $(\l-1)$-homogeneous solutions in fact possess a stream-function. Indeed, using the divergence-free condition $u_1' + \l u_2 = 0$, and setting $\Psi(r,\th) = r^{\l} \psi(\th)$ with $\psi(\th) = u_1/ \l $ we obtain 
$$
\bu = \n^\perp \Psi =  r^{\l-1}\left[   \l \psi \btau - \psi' \bnu  \right].
$$
The equation \eqref{ee-polar}
then reads
\begin{equation}\label{ODE}
\begin{split}
2 (\l-1) P & =- (\l-1) (\psi')^2 +\l^2 \psi^2 + \l \psi'' \psi,\\
\psi(0) & = \psi(2\pi).
\end{split}
\end{equation}
For locally finite energy solutions, $u\in L^2_{\loc}(\R^2)$, which are of our primary concern, we have $\psi \in H^1(\T)$. The equation  \eqref{ODE} under such low regularity  can be understood in the distributional sense. That will be detailed in \lem{l:ODEdist-q} below. Let us first make some formal observations. There are three obvious symmetries of the equation:
\begin{itemize}
\item[(i)] Rotation: $\psi \ra \psi(\cdot - \th_0)$.
\item[(ii)] Scaling: $(P,\psi) \ra (\a^2 P,\a\psi)$.
\item[(iii)] Reflection: $\psi \ra \psi(-\th)$.
\end{itemize} 
Let us list some explicit solutions to \eqref{ODE} up to the symmetries noted above. Those that have singularities or exist only on part of the circle should at this point be viewed as local solutions on their intervals of regularity.
\begin{align}
\psi& \equiv \frac{\sqrt{2(\l-1)P}}{\l}, (\l-1)P>0 & &\text{rotational flow  } \Psi =  \psi r^\l, \label{ex:rot}\\
\psi& = A |\cos(\th)|^\l, P=0 & &\text{parallel shear flow  } \Psi(x,y) =  A|x|^\l, \label{ex:psf}\\ 
\psi& = \g_1 + \g_2 \cos(2\th), \l = 2, P = 2(\g_1^2 - \g_2^2) & &\Psi = (\g_1+\g_2)x^2+ (\g_1-\g_2)y^2, \label{ex:2}\\
\psi &= \sqrt{\g_1 + \g_2 \cos(\th)}, |\g_2| \leq \g_1, \l = \frac12, P =-\frac{\g_1}{4}  & & \Psi = \sqrt{\g_1 r + \g_2 x},  \label{ex:half}\\
\psi & = \frac{\sqrt{-2P}}{\l}\cos(\l \th ), P<0, \l\geq \frac12.  & &\label{ex:ho}
\end{align}
Observe that in example \eqref{ex:2} elliptic, parallel, and hyperbolic configurations of stream-lines are determined by $P>0$, $P=0$, $P<0$, respectively. In example \eqref{ex:half} the case $|\g_2|<\g_1$ produces non-vanishing $\psi$ with elliptic steam-lines, case $|\g_2|=\g_1$ yields parabolic steam-lines, and  the case $|\g_2|>\g_1$ yields hyperbolic steam-lines in the segments $\{ -\g_2 \cos \th  < \g_1 \}$, although these solutions don't belong to $H^1$, therefore we don't list them.  Example \eqref{ex:ho} is $2\pi$-periodic only for $\l \in \frac12 \N$, although it is formally a solution to \eqref{ODE}. We will see that all values $\l\geq \frac12$ are relevant, as we will be able to piece together local solutions where they vanish to reconstruct complete solutions. For $\l>\frac12$ solutions have hyperbolic streamlines, case $\l = \frac12$ is a subcase of \eqref{ex:half} and is parabolic. For $\l<\frac12$ the same solution \eqref{ex:ho} is sign-definite on a period longer than $2\pi$, therefore cannot be used. 

Generally, elliptic-type solutions correspond to non-vanishing streamfunction $\psi$, hyperbolic-type is described by $\psi$ vanishing at two or more points unless it is a parallel shear flow, and parabolic-type corresponds to vanishing at one point only.

Occurrence of singular solutions makes it necessary to clarify the relationship between \eqref{ODE} and the original Euler equations in the weak settings.

\begin{lemma}\label{l:ODEdist-q} Let $I \ss \T$ be an open interval. Suppose $\psi \in H^1(I)$. Set $\bu =  r^{\l-1}\left[   \l \psi \btau - \psi' \bnu  \right]$ and $p = r^{2\l -2}P$, where $P \in \cD(\T)$. Then
\begin{equation}\label{pq}
\bu \cdot \n \bu + \n p = 0
\end{equation}
holds in the weak sense in the sector $\Sigma = \{ \th \in I, 0<r < \infty \}$, which means that  for all $\bv \in C^\infty_0(\Sigma)$ one has
\begin{equation}\label{ee-weak}
\int_\Sigma \bu \cdot \n \bv \cdot \bu dx + \int_0^\infty P( \n \cdot \bv(r,\cdot)) r^{2\l - 1} dr = 0,
\end{equation}
if and only if $P$ is a constant and the identity
\begin{equation}\label{ODEdist-q}
2(\l-1) P = - (2\l-1)(\psi')^2 + \l^2 \psi^2 - \l \psi \psi' \p_\th 
\end{equation}
holds in the distributional sense on $I$. Consequently, \eqref{ODEdist-q} holds on $\T$ if and only if \eqref{ee-weak} holds on $\R^2\bs \{0\}$. If $\l >  \frac12$, then \eqref{ee-weak} holds on the whole space $\R^2$.
\end{lemma}
\begin{proof}  Suppose  \eqref{ee-weak} holds. Plugging $\bv =f(r)[ A(\th)\btau +B (\th)\bnu]$ with $f\in C^\infty_0(0,\infty)$, $A,B \in C^\infty_0(I)$ one can see that the $r$-integrals separate and cancel from the equation. The $\th$-integrals give
\begin{equation}\label{}
\begin{split}
&\int_\T [ \l^2 \psi^2 A' + \l \psi' \psi A+(2\l-1)\l \psi'\psi A ] \, d\th \\
&+ \int_\T[ \l^2 \psi^2 B - \l \psi'\psi B' -(2\l-1)(\psi')^2 B]\, d\th \\
&= -2(\l-1) P(B) + P(A').
\end{split}
\end{equation}
The integral with $A$ on the left hand side vanishes. Plugging $B =0$ shows that $P(A')=0$ for any $A$. Consequently, $P$ is a constant function. Reading off the terms involving $B$ leads to \eqref{ODEdist-q}. The converse statement is routine. Finally, the last statement is verified by 
taking $\bv \s_\e$, where $\s_\e(r) = \s(r/\e)$, $\s = 1$ for $r>1$, $\s =0$, $r<1/2$, and passing to the limit as $\e \ra 0$.
\end{proof}

\subsection{Life-times, general structure, special cases} Let us suppose that for some open interval $I = (a,b)$, $\psi\in H^1(I)$ (note that $\psi$ is automatically continuous), $\psi$ solves \eqref{ODEdist-q} and is sign-definite on $I$. Then 
\[
\psi' \p_\th = \frac{- 2(\l-1) P - (2\l-1)(\psi')^2 + \l^2 \psi^2 }{-\l \psi},
\]
and hence $\psi'' \in L^1(I')$ for any compactly embedded subinterval $I' \ss I$. By bootstraping on the regularity, we conclude that $\psi \in C^\infty$ inside $I$.   This is of course the standard elliptic regularity conclusion. So, $\psi$ can only lose smoothness where it vanishes, otherwise it satisfies \eqref{ODE} classically. To fix the terminology, if  $|\psi| \neq 0$ on the interval $(a,b)$, and $\psi(a) = \psi(b)=0$, then $T = b-a$ will be called the \emph{life-span} of $\psi$ and $(a,b)$ its \emph{life-time}. By the standard uniqueness for ODEs, if life-times of two solutions $\psi_1$ and $\psi_2$ overlap, and $\psi_1(\th_0) = \psi_2(\th_0)$, $\psi'_1(\th_0) = \psi'_2(\th_0)$ at some point $\th_0$, then $\psi_1 \equiv \psi_2$ and their life-times coincide. There are a few immediate consequences of this uniqueness, and in some cases it allows us to give a complete description of solutions to \eqref{ODE}. We discuss it next.

Let us fix solution $\psi\in H^1(I)$ with $I$ being a life time of $\psi$. Suppose $I$ is not the whole $\T$. Then $\psi$ vanishes at its end-points, hence there exists $c\in I$ where $\psi'(c) = 0$. By the reflection and rotation symmetries, $\psi_1(\th) = \psi(c-\th)$ and $\psi_2(\th) = \psi(c+\th)$ are two solutions of the same equation \eqref{ODE} on an open neighborhood of $0$ with the same initial data $(\psi_1,\psi_1')=(\psi_2,\psi_2')$ at $\th = 0$. By uniqueness, it follows that $\psi$ is symmetric with respect to $c$, and $c$ is the middle point of $I$. This also proves that $c$ is the only critical point of $\psi$ on $I$ (for otherwise $I$ covers the entire circle). Thus, any vanishing solution to \eqref{ODE} is an arch-shaped function. The streamlines of the velocity field in this case go off to infinity at the edges of the corresponding sector.  
\begin{definition} A pair $(\psi, I)$, where $\psi \in H^1(I)$ solves  \eqref{ODE} on $I$ and $I$ is a life-time of $\psi$ will be called a local solution of \eqref{ODE}.
\end{definition}
In some cases we can give a complete and explicit description of local and even global solutions.

\begin{lemma}\label{l:psf} The following statements are true.
\begin{itemize}
\item[(a)] For $\l=1$ or $P=0$ any solution is a parallel shear flow with stream-function given by  (up to symmetries)
\begin{equation}\label{sf:psf}
\Psi = |x|^\l(A  \chi_{\{x>0\}} + B \chi_{\{x\leq 0\}}),\, A,B \in \R.
\end{equation}
\item[(b)] For $\l=2$ all local solutions $(\psi,I)$ are given by \eqref{ex:2} up to rotation. In this case life-spans may range from $0$ to $2\pi$. 

\item[(c)] For $\l = \frac12$ all local solutions $(\psi,I)$ are given by \eqref{ex:half}. So, only elliptic and parabolic solutions are possible with life-spans equal $2\pi$.
\end{itemize}
\end{lemma}
\begin{proof}
If $\l=1$, then the pressure $p=P$ is constant, and it can be chosen $0$. So, the case $\l=1$ is a subcase of $P=0$. Let $P=0$, and let $\psi$ be a solution and let $I$ be a life-time of $\psi$. By symmetry we can assume that $\psi>0$ on $I$ and $0\in I$ is the middle point. So $\psi'(0) = 0$. If $\psi(0) = A$, then we can see that example \eqref{ex:psf} gives another solution with zero pressure and the same initial data at $0$. This implies that the time-span of $I$ is $\pi$, and $\psi$ itself is given by \eqref{ex:psf}. Making the same argument on the complement of $I$ we arrive at the conclusion. Alternatively, we can argue with a direct use of the Euler equation.  Let us consider an interval $I$ where $\psi$ is sign-definite. Since $\psi$ is smooth on $I$, in the sector $\Sigma=\{ 0<r<\infty, \th \in I\}$ the Euler equation can be written classically, $\bu\cdot \n \bu = 0$. This becomes the geodesic equation for particle trajectories. So, in $\Sigma$, $\bu$ is a parallel shear flow. By homogeneity, the only possibility is then given by the stream-function \eqref{sf:psf} up to a rotation and scalar multiple. This forces $I$ to be of length $\pi$. Similar argument applies to the complement of $I$, unless $\psi$ vanishes identically there. In either case, $\Psi$ is given by \eqref{sf:psf}.

Now let $\l=2$ and $\psi$ be a solution corresponding to pressure $P \in \R$ and with a life-time $I$. Without loss of generality, let $I$ be centered at the origin. Denoting $\psi(0) = A$ we can solve the system $A = \g_1+\g_2$, $P/2A = \g_1-\g_2$ to find $\g_1$, $\g_2$. Setting  $\tilde{\psi}$ to be as in \eqref{ex:2} gives another solution with the same pressure and initial data at $0$. Hence, $\tilde{\psi} = \psi$. 

Let us now consider the case $\l = \frac12$. Suppose as before $\psi > 0$ with local maximum at $0$. Define $\g_1 = -4P$, and let $\g_2> - \g_1$ be determined by $\sqrt{\g_1+\g_2} = \psi(0)$. Then the solution given by \eqref{ex:half} is another one with the same initial data at $0$, hence the two coincide. The boundedness in $H^1$ necessitates $\g_1>0$ and $|\g_2| \leq \g_1$. So,  \eqref{ex:half} gives a complete description of solutions in this case. 

\end{proof}

\begin{lemma}[General Structure]\label{l:structure} 
Let $\l>0$, $P\in \R$, and $P\neq 0$, $\l \neq 1$.
\begin{itemize}
\item[(i)] Suppose $\psi\in H^{1}(\T)$ is a weak solution to  \eqref{ODEdist-q}. Then there is a collection of disjoint intervals $I_n$, $n=1,2,...$ such that $|\cup_{n=1}^\infty I_n| = 2\pi$ and $(\rest{\psi}{I_n}, I_n)$ is a local solution for each $n$.
\item[(ii)] Conversely, let $\{(\psi_n,I_n)\}_{n=1}^\infty$ be a collection of local solutions corresponding to the same values of $P$ and $\l$, with $I_n$'s disjoint and $|\cup_{n=1}^\infty I_n| = 2\pi$. Then the globally defined function $\psi = \sum_{n=1}^\infty \psi_n \chi_{I_n}$ belongs to $H^1(\T)$ and is a distributional solution to \eqref{ODEdist-q} on the whole circle $\T$.
\end{itemize}
\end{lemma}

We see that the case of zero pressure (or $\l=1$) is the only case when a solution can vanish on an open set. Otherwise, life-times must fill a set of full measure. Moreover local solutions with their life-times serve as building blocks. We can rearrange and piece them together to form new solutions.

\begin{corollary}
Suppose $\psi \in H^1(\T)$ is a solution to \eqref{ODEdist-q}, $|\cup_{n=1}^\infty I_n| = 2\pi$,  and $\psi = \sum_{n=1}^\infty \psi_n \chi_{I_n}$, where $I_n$ is a life-time of $\psi_n$. Then for any spatial rearrangements of intervals $\{J_m\}_m = \{I_n\}_n$, corresponding rearrangement of functions $\{\phi_m\}_m = \{\psi_n\}_n$, and any distribution of signs $\e = \{\e_m = \pm 1\}_m$, the function $\phi = \sum_{m=1}^\infty \e_m \phi_{m} \chi_{J_m}$ is a solution to \eqref{ODEdist-q}.
\end{corollary}

\begin{proof}[Proof of \lem{l:structure}]
(i). The set $\{\psi \neq 0\}$ is open, so it is a union of disjoint open life-time interval $\{I_n\}$. To proceed we first examine the end-point behavior of $\psi$ on one interval $I = (a,b)$. We have $\psi \in C^\infty(I)\cap H^1(I)$, and $\psi (a) = \psi(b) = 0$. We have the Newton's formula, $\psi(\th) = \int_a^\th \psi' dy$, and hence the bounds 
\begin{equation}\label{near}
|\psi(\th) | \leq \sqrt{\th-a} \| \psi\|_{H^1},\quad  |\psi(\th) | \leq \sqrt{b-\th} \| \psi\|_{H^1}.
\end{equation}
Let $c \in I$ be fixed and $\e>0$ be small. From integrating \eqref{ODEdist-q} on $[c,b-\e]$ we find
\begin{equation}\label{e:compbyparts}
(b-\e-c) 2(\l-1) P = \int_c^{b-\e}[- (2\l-1)(\psi')^2 + \l^2 \psi^2]d\th + \l \psi \psi'(b-\e) - \l \psi \psi'(c),
\end{equation}
and a similar identity holds near $a$. This implies that the one-sided limits $(\psi \psi')(a^+)$, $(\psi \psi')(b^-)$ exist and are finite. If $(\psi \psi')(a^+) \neq 0$, then in view of \eqref{near}, $|\psi'(\th)| \gtrsim \frac{1}{\sqrt{\th-a}}$, contradicting our assumption $\psi \in H^1$. Similarly, $(\psi \psi')(b^-) = 0$.  Thus, the function $\psi \psi'$ vanishes at the birth and death-times of $\psi$.

Note that $\psi' = 0$ a.e. on the set $\{\psi = 0\} = \T \bs \cup_n I_n$. With this in mind, let us test \eqref{ODEdist-q} against $1$ on $\T$, and use the fact that on each $I_n$ our function solves the classical ODE \eqref{ODE},
\[
\begin{split}
4\pi(\l-1)P &= \sum_n \int_{I_n}[ -(2\l-1) (\psi')^2 + \l^2 \psi^2]\, d\th \\
&= \sum_n \int_{I_n}[ -\l (\psi')^2 - \l \psi'' \psi +2(\l-1)P]\, d\th \\
\intertext{and integrating by parts using that $\psi\psi'$ vanishes at the end-points,}
&
= 2(\l-1)P \left| \cup_{n=1}^\infty I_n \right|.
\end{split}
\]
Hence, $\left| \cup_{n=1}^\infty I_n \right| = 2\pi$, provided $(\l-1)P \neq 0$.

(ii). Let us show that $\psi\in H^1(\T)$ first. In the case $\l = \frac12$ as we know from \lem{l:psf}(c), there is only a single $I_n$ with $|I_n|=2\pi$, which makes the conclusion trivial. Suppose $\l \neq 1/2$.  Let us integrate \eqref{ODE} over one $I_n$. We have 
\[
(2\l-1) \int_{I_n} (\psi'_n)^2 d\th = -2(\l-1)P |I_n| + \l^2 \int_{I_n} (\psi_n)^2\, d\th.
\]
Since $\psi$ vanishes at the end-points of $I_n$, by the Poincare inequality
\[
\int_{I_n} (\psi_n)^2 d\th \leq |I_n|^2 \int_{I_n} (\psi'_n)^2 d\th.
\]
So, $\int_{I_n} (\psi'_n)^2 d\th \leq C |I_n|$,
for some $C>0$ and all $n$. Thus, $\int_\T (\psi')^2 d\th \leq C\sum_n |I_n| <\infty$. This proves that $\psi \in H^1(\T)$.

To show that $\psi$ is a global solution to \eqref{ODEdist-q}, let us fix a test-function $\f$ on $\T$. Due to vanishing of $\psi\psi'$ at the end-points and that $\psi$ is classical on each $I_n$, we obtain
\begin{equation*}\label{}
\begin{split}
\int_{\T}[ - (2\l-1)(\psi')^2 + \l^2 \psi^2 - \l \psi \psi' \p_\th ] \f d\th &=
\sum_n \int_{I_n}[ - (2\l-1)(\psi')^2 + \l^2 \psi^2 - \l \psi \psi' \p_\th ] \f d\th \\
&= 
\sum_n \int_{I_n} [- (\l-1)(\psi')^2 + \l^2 \psi^2 + \l \psi \psi''] \f d\th \\
& =
\sum_n \int_{I_n}  2(\l-1)P \f d\th = \int_\T  2(\l-1)P \f d\th .
\end{split}
\end{equation*}
Thus, $\psi$ is a solution on $\T$.
\end{proof}

We thus have reduced the classification problem to the problem of finding local positive solutions $(\psi,I)$ and determining existence of life-spans that add up to $2\pi$.

\subsection{Bernoulli's law and its direct consequences} Let $(\psi,I)$ be a local solution with $\psi >0$ on $I$.
There is a natural conservation law associated with \eqref{ODE} that can be derived from the conservation of the Bernoulli function on streamlines. Writing the Euler equation in the form 
\begin{equation}\label{eeB}
\bu^\perp \w + \n (2 p+ |\bu|^2) = 0,
\end{equation}
where $\w =  r^{\l-2}( \l^2 \psi + \psi'')$ is the scalar vorticity, we obtain $\bu \cdot \n (2 p+ |\bu|^2) = 0$. Denoting the Bernoulli function on the unit circle $\b = 2P + (\psi')^2 + \l^2 \psi^2$, the above equation in polar coordinates reads
\[
\l \psi \b' = 2(\l-1) \psi' \b.
\]
Integrating we obtain $\b = c \psi^{\frac{2\l - 2}{\l}}$, for some constant $c$. We have found the following  conservation law:
\begin{equation}\label{B}
B = (2P+\l^2 \psi^2+(\psi')^2)\psi^{\frac{2}{\l}-2}.
\end{equation}
It is important to note that $B$ remains constant only on a life-time of $\psi$. If $\psi$ is a global solution, $B$, unlike the pressure $P$, may change from life-time to life-time. Therefore, $B$ serves to parametrize local solutions corresponding the same pressure.  Plugging \eqref{B} into \eqref{ODE} gives the following equation
\begin{equation}\label{BODE}
\frac{\l-1}{\l} \psi^{1-\frac{2}{\l}} B  = \l^2 \psi + \psi''.
\end{equation}
We note that \eqref{BODE} can be alternatively obtained from the vorticity formulation of the Euler equation: $\bu \cdot \n \w =0$. It is clear that example \eqref{ex:ho} corresponds to $B=0$, where \eqref{BODE} becomes a harmonic oscillator. Next we show that signs of $B$ and $P$ determine the type of a solution and exclude some types. Let us observe one simple rule, which follows directly from \eqref{B}
\begin{equation}\label{PtoB}
B<0 \Rightarrow P<0.
\end{equation}

\begin{lemma} \label{l:signs} For $0<\l<1$ a solution is elliptic if and only if $B<0$. For $\l>1$ a solution is elliptic if and only if $P>0$.
\end{lemma}
\begin{proof}

Let $\l<1$ and $\psi >0$ be a non-vanishing solution. At its local minimum $\psi'' \geq 0$, $\psi >0$. So, the right hand side of  \eqref{BODE} at this point is positive, hence $B<0$. If, on the other hand, $\psi$ vanishes at $\th =a$, then letting $\th \ra a$ in \eqref{B} we find that, unless $B=0$, $\psi'(\th) \ra \infty$, in which case $B>0$.  Thus, $B\geq 0$.

Now let $\l>1$ and $\psi$ be elliptic. Evaluating \eqref{ODE} at a local minimum reveals $P>0$. If, on the other hand, $\psi$ vanishes, multiplying \eqref{BODE} with $\psi$ and letting $\th \ra a$, where $\psi(a) =0$ we see that $\psi \psi'' (\th) \ra 0$. Then letting $\th \ra a$ in \eqref{ODE} shows that the limit of $(\psi')^2(\th)$ exists and is equal to $-2P$, hence $P \leq 0$. 

\end{proof}

As we see from the proof, in case $\l>1$, all vanishing solutions hit the origin at the same slope up to a sign. Thus, ``gluing" local oppositely signed solutions with the same pressure produces a smooth $C^1$ connection. This point will be elaborated upon later. Next we exclude certain types of solutions with the help of $B$ (here we only discuss cases not classified by \lem{l:psf})

\begin{lemma}\label{l:ellex} For $0<\l < \frac12$ all $H^1$-solutions are elliptic. For $\l>\frac12$ all vanishing solutions belong to $H^1$ automatically.
\end{lemma}
\begin{proof} First, let $0<\l<\frac12$ and let $(\psi,I)$ be a local vanishing solution. If $B=0$ on $I$, then $\psi$ is given by \eqref{ex:ho} with period exceeding $2\pi$ in the given range of $\l$. So, $B \neq 0$. Then $(\psi')^2 = B\psi^{2- \frac{2}{\l}} - 2P- \l^2 \psi^2$. In view of \eqref{near} and since $\l <1$ we have $(\psi')^2(\th) \gtrsim (\th-a)^{1- \frac{1}{\l}}$. In order to guarantee that $\psi \in H^1(I)$, $\l$ must exceed $\frac12$ strictly, in contradiction with the assumption. 

For $\l>1$,  we already proved that $(\psi')^2 \ra -2P$, and thus finite, at any point where $\psi$ vanishes. So, $\psi \in H^1$. Now, for $\frac12<\l<1$, we have $(\psi'\psi)^2 = B\psi^{4- \frac{2}{\l}} - 2P\psi^2- \l^2 \psi^4$. Thus, the product $\psi'\psi$ remains bounded and in fact vanishes where $\psi$ does. The formula \eqref{e:compbyparts} readily implies that $\psi' \in L^2(I)$. 
\end{proof}

\subsection{Hamiltonian structure, rescaling and conjugacy} According to \lem{l:structure}, the classification of solutions reduces to the question of describing local pairs $(\psi,I)$, where $\psi \in H^1(I)$, and the issue of existence of sequences of life-spans that add up to $2\pi$. So, let $\psi >0$ on $I$ and $I$ is the life-time of $\psi$. Rewriting equation \eqref{BODE} in phase variables $(x,y) = (\psi,\psi')$ gives the system
\begin{equation}\label{sys-B}
\left\{\begin{split}
		x' & = y\\
		y' & = -\lambda^2x +\frac{\lambda-1}\lambda B x^{\frac{\lambda-2}\lambda}.
\end{split}\right.
\end{equation}
Considering $B$ fixed, the system is to be considered on the right half-plane $\{ x\geq 0\}$, and it has a Hamiltonian given by the pressure
\[
 P = - \frac{y^2}{2} - \frac{\l^2}{2} x^2 + \frac{B}{2} x^{\frac{2\l-2}{\l}}.
\]

We now have two available variables $P$ and $B$ to parametrize the family of solutions. The phase portrait of the system changes character depending on the signs of $P$ and $B$ and on location of $\l$ with respect to the pivotal value $\l=1$. We can however rescale the system in several ways to reduce the number of possibilities.
First, all solutions with equal sign of pressure $P$ can be transformed into one with $|P|=1$ by
\begin{equation}\label{scaleP}
\tilde{x} =   \frac{x}{\sqrt{|P|}}, \quad \tilde{y} = \frac{y}{\sqrt{|P|}},\quad \tilde{P} =\sign(P), \quad \tilde{B} = \frac{B}{|P|^{1/\l}}.
\end{equation}
Alternatively, we can rescale the Bernoulli parameter $B$:
\begin{equation}\label{scaleB}
\tilde{x} = \frac{x}{|B|^{\l/2}}, \quad \tilde{y} = \frac{y}{|B|^{\l/2}},\quad \tilde{B} = \sign(B),\quad \tilde{P} = \frac{P}{|B|^\l}.
\end{equation}
These rescalings do not change life-times and spans of solutions. Another important symmetry relates cases $0<\l<1$ and $1<\l$ by a conjugation transformation. Let $\l >0$, and let $(x,y)$ be a solution to \eqref{sys-B} corresponding to a pair $(P,B)$. A routine verification reveals that the new pair $(\tilde{x},\tilde{y})$ given by
\begin{equation}\label{e:conj}
\tilde{x}(t) = {x}^{1/\l}(t/\l), \quad \tilde{y} = \tilde{x}'.
\end{equation}
solves the same system \eqref{sys-B} now with 
\begin{equation}\label{}
\tilde{\l} = \frac{1}{\l}, \quad \tilde{B} = - \frac{2P}{\l^4}, \quad \tilde{P} = - \frac{B}{2\l^4}.
\end{equation}
Denoting by $T_\l(P,B)$ the life-span of $(x,y)$, the spans of conjugate solutions are related by
\begin{equation}\label{}
T_\l(P,B) = \tilde{\l} T_{\tilde{\l}} (\tilde{P}, \tilde{B}).
\end{equation}
Clearly, the conjugacy is the reason behind the apparent duality between the statements of \lem{l:signs} (ii) and (iii), and examples \eqref{ex:2} and \eqref{ex:half}. Although we find it instructive to obtain those directly from the equations.
Using these rescalings and conjugation we can reduce our considerations to the case $B=\pm 1,0$, and $\l>1$, keeping in mind the precautions of \lem{l:signs}. The question now becomes which solutions, if any, have time-spans that add up to $2\pi$, and which type of solutions exist for various values of the pressure.

\section{Description of solutions for $\l>1$} \label{s:desc}
\begin{figure}\label{connection}
        \centering
           \includegraphics[width=4in]{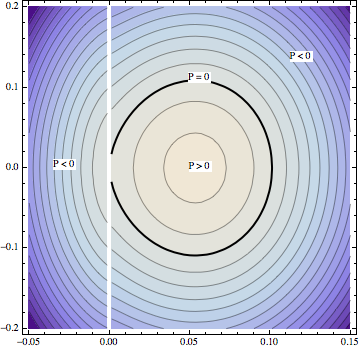}
                \caption{Here $\l = 2.5$. The right half plane shows phase portrait for B=1, the left half shows portrait for $B=-1$ (with sign of $x$ reversed).}
                \label{connection}
\end{figure}
The system \eqref{sys-B} exhibits qualitatively similar phase portraits for all $\l>1$, see Figure~\ref{connection} for $\l = 2.5$. According to \lem{l:signs}, $P>0$ (hence $B>0$) defines the elliptic region inside the separatrix. In that region there is a steady state:
\[
x_s =  \left(  \frac{\l-1}{\l^3} \right)^{\frac{\l}{2}}, \quad y_s = 0,
\]
corresponding to the maximal value of the pressure  $P_{\max} =  \frac{1}{2\l} \left( \frac{\l-1}{\l^3} \right)^{\l-1}$. From that point out the pressure decreases to value $P=0$, which corresponds to the parallel shear flow. The level curve $P=0$ defines the homoclinic separatrix. As  $P$ turns negative,  we observe hyperbolic solutions that split into three kinds: $B<0$, $B>0$, $B=0$. The right hand side of Figure~\ref{connection} depicts region corresponding to $B=1$, while on the left (with sign of $x$ reversed for visual comparison) corresponding to $B=-1$. The solution corresponding to $B=0$ on this portrait ``lives" at infinity, which can be obtained in the limit as $P \ra -\infty$ after rescaling given by \eqref{scaleP}. It is explicitely given by \eqref{ex:ho}. 

In the elliptic region, the problem reduces to simply finding trajectories with period $2\pi$, while in the hyperbolic region we seek solutions whose life-spans can add up to $2\pi$. So, let us examine the end-point behavior of $T_\l(P,1)$. At the steady state, $T_\l = 0$ of course, however as $P \ra P_{\max}$, the periods approach that of the linearized system by the standard perturbation argument. So, in this sense $T_\l(P_{\max}, 1) = \frac{2\pi}{\sqrt{2\l}}$. As $P\ra 0+$, we have $T_\l(0,1) = \pi$ as it is the life-span of the parallel shear flow. In this case more work is necessary to prove convergence as the system \eqref{sys-B} losses regularity near the origin.  On the hyperbolic side, we expect $T_\l(P,1) \ra \pi$ as $P \ra 0-$, and $T_\l(P,1) \ra \frac{\pi}{\l}$ as $P \ra -\infty$ as the latter is the life-span of harmonic oscillator \eqref{ex:ho}. When $B=-1$ we expect $T_\l(P,-1) \ra \frac{\pi}{\l}$ as $P \ra -\infty$ for the same reason, but as $P\ra 0-$ the life-span is expected to vanish since the parallel shear flow produces $B>0$. 

The next question is whether the period or life-spans are monotonic functions in the ranges of $P$ as above. Monotonicity allows us to count exactly how many $2\pi$-periodic solutions exist in elliptic case or in hyperbolic case to gives the exact range of the life-span function. Let us state our main result.

\begin{proposition}\label{p:mon}  
\begin{itemize}
\item[(i)] Elliptic case: let $\l\geq \frac43$. The period-function $T_\l(P,1)$ changes monotonely from $\frac{2\pi}{\sqrt{2\l}}$ to $\pi$, as $P$ decreases from $P_{\max}$ to $0$. 
\item[(ii)] Hyperbolic case: let $\l>1$. As $P$ passes from $0$ to $-\infty$ the life-spans $T_\l(P,1)$ decrease monotonely from $\pi$ to $\frac{\pi}{\l}$, while $T_\l(P,-1)$ increase from $0$ to $\frac{\pi}{\l}$.
\end{itemize}

\end{proposition}
The elliptic case in the range $1< \l < \frac43$ remains unresolved at this moment, although the proof of the convergence to the limit values of $\frac{2\pi}{\sqrt{2\l}}$ and $\pi$ still applies. We will comment further on this case at the end.

\subsection{Proof in the hyperbolic case}  In this case both signs of $B$ are possible. Let $B = 1$. The level curves $-2P = y^2 +\l^2 x^2 - x^{2-2/\l}$ determine the orbits of the system. Let $x_0$ be the $x$-intercept. We thus have
\[
y = \pm \sqrt{\l^2(x_0^2-x^2) - (x_0^{2-2/\l}- x^{2-2/\l})}.
\]
Integrating over the positive half and changing the variable to $\xi = x/x_0$ we obtain
\begin{equation}\label{T:ell}
\frac12 T = \int_0^{1} \frac{1}{\sqrt{\l^2(1-\xi^2) - x_0^{-2/\l}(1 - \xi^{\frac{2\l-2}{\l}})}} \, d\xi.
\end{equation}
Let $f(x_0,\xi)$ be the integrand. By a direct computation $\p_{x_0} f <0$ pointwise for all $0<\xi<1$. Thus, the life-span function decreases. Furthermore, 
\[
\frac{1}{\l \sqrt{1-\xi^2}}< f(x_0,\xi) < \frac{1}{\l \sqrt{\xi^{2-2/\l} - \xi^2}} ,
\]
with the bounds being achieved in the limits as $x_0 \ra 1/\l^\l$ and $x_0\ra \infty$, respectively. Integrating in $\xi$ recovers the limit values $T \ra \pi$ and $T \ra \frac{\pi}{\l}$ as desired.

If $B = -1$, the argument is similar.  We have in this case
\[
\frac12 T = \int_0^{1} \frac{1}{\sqrt{\l^2 (1-\xi^2) + x_0^{-2/\l}(1 - \xi^{\frac{2\l-2}{\l}})}} \, d\xi, \quad \p_{x_0} f >0,
\]
and
\[
0 < f(x_0,\xi) <  \frac{1}{\l \sqrt{1-\xi^2}},
\]
with the bounds being achieved in the limits as $x_0 \ra 0$ and $x_0\ra \infty$, respectively. Thus, $T \ra 0$ and $T \ra \frac{\pi}{\l}$ as desired.

\subsection{Proof in the elliptic case} The proof in this case is more involved and will be split into several  parts. Let us set $B=1$ and denote $T(P) = T_\l(P,1)$.

\subsubsection{Convergence to $\pi$} For a fixed $P>0$, let $x_0$ and $x_1$ be the two points of intersection of the corresponding orbit with the $x$-axis. We have
\[
y = \pm \sqrt{x^{2-2/\l} - \l^2 x^2 - (x_{0,1}^{2-2/\l} - \l^2 x_{0,1}^2)}.
\]
In view of the symmetry,
\[
\frac12 T(P) = \int_{x_0}^{x_1} \frac{dx}{\sqrt{ x^{2-2/\l} - \l^2 x^2 - (x_{0,1}^{2-2/\l} - \l^2 x_{0,1}^2)}}.
\]
Let us change variable to $\xi = \l x^{1/\l}$, then
\[
\frac12 T(P) = \int_{\xi_0}^{\xi_1} \frac{d\xi}{\sqrt{ 1-\xi^2 - \d \xi^{2-2\l} }},
\]
where $\d = \xi_{0,1}^{2\l-2}(1-\xi_{0,1}^2) \ra 0^+$. Let us fix a small parameter $\e >0$ and an exponent
 \begin{equation}\label{}
\frac{\l - 3/2}{\l+1} < \a < 1,
\end{equation}
and split the integral as follows
\[
\int_{\xi_0}^{\xi_1} = \int_{\xi_0}^{\xi_0^\a} + \int_{\xi_0^\a}^\e + \int_\e^{1-\e} + \int_{1-\e}^{\xi_1} = A +B +C +D.
\]
We will be taking two consecutive limits, first as $\d \ra 0$, and then as $\e \ra 0$. First note, that the middle integral $C$ is proper for small $\d$, thus
\[
C \underset{\d \ra 0}{\to}  \int_\e^{1-\e}\frac{d\xi}{\sqrt{ 1-\xi^2 }} = \sin^{-1}(1-\e) - \sin^{-1}(\e) \underset{\e \ra 0}{\to}  \frac{\pi}{2}.
\]
So, it remains to show that the other integrals converge to zero. We have
\[
\frac{\d}{\xi_0^{2\l-2}} \to 1, \text{ as } \d \ra 0.
\]
Using this we estimate
\[
\begin{split}
B  \leq \int_{\xi_0^\a}^\e  \frac{d\xi}{\sqrt{ 1-\frac{\d}{\xi_0^{\a(2\l-2)}}-\xi^2 }} &=
\sin^{-1} \left( \frac{\e}{\sqrt{1-\frac{\d}{\xi_0^{\a(2\l-2)} }}} \right)- \sin^{-1} \left( \frac{\xi_0^\a}{\sqrt{1-\frac{\d}{\xi_0^{\a(2\l-2)} }}}\right)\\
&\underset{\d \ra 0}{\to}\sin^{-1}(\e) \underset{\e \ra 0}{\to} 0.
\end{split}
\]
Denote $f(\xi) = \xi^{2\l-2}(1-\xi^2)$. Thus, $\d = f(\xi_0) = f(\xi_1)$. In terms of $f$ we have
\[
A = \int_{\xi_0}^{\xi_0^\a} \frac{ \xi^{\l-1} \, d\xi}{\sqrt{ f(\xi) - f(\xi_0)}}.
\]
We consider two cases $\l >3/2$ and $\l \leq 3/2$. In the case $\l >3/2$, function $f$ is convex and increasing near the origin. Thus, we have inequality
\[
f(\xi) - f(\xi_0) \geq f'(\xi_0)(\xi-\xi_0) \sim \xi_0^{2\l-3}(\xi-\xi_0).
\]
Substituting into the integral, we obtain
\[
A \lesssim \frac{\xi_0^{\a(\l-1)}}{\xi_0^{\l - 3/2}} \int_{\xi_0}^{\xi_0^\a} \frac{d\xi}{\sqrt{\xi-\xi_0}} 
\sim \frac{\xi_0^{\a(\l-1)}\xi_0^{\a/2}}{\xi_0^{\l - 3/2}} = \xi_0^{\a(\l+1) + \frac{3}{2} - \l} \ra 0,
\]
in view of our condition on $\a$. If $\l \leq 3/2$, then $f$ is concave and still increasing in the vicinity of the origin, and hence,
\[
f(\xi) - f(\xi_0) \geq f'(\xi_0^\a)(\xi-\xi_0) \sim \xi_0^{\a(2\l-3)}(\xi - \xi_0).
\]
Substituting, we similarly obtain
\[
A \lesssim \frac{\xi_0^{\a(\l-1)}}{\xi_0^{\a(\l - 3/2)}} = \xi_0^\a \ra 0.
\] 
For the remaining integral $D$, we notice that in the vicinity of $1$ we have
\[
|f(\xi) - f(\xi_1)| \sim |f'(\xi_1)||\xi-\xi_1| \sim |\xi-\xi_1|.
\]
Thus, 
\[
D \sim \int_{1-\e}^{\xi_1} \frac{d\xi}{\sqrt{\xi_1-\xi}} \sim \sqrt{\xi_1 - 1+\e} \ra \sqrt{\e} \ra 0.
\]
This completes the proof.

\subsubsection{Monotonicity of the period: setup} We now return to the original Hamiltonian system \eqref{sys-B} with $B=1$ as before. Let us recall the sufficient condition of monotonicity given in \cite{chicone,cima}. Suppose $H(x,y)$ is a Hamiltonian of a 2D system of ODEs with $(x_0,0)$ being a non-degenerate minimum, $H= 0$. Let us suppose $H = \frac{y^2}{2}+ V(x) $. The period-function $T= T(h)$ as a function of level sets $H = h$ is increasing if 
\begin{equation}\label{e:V}
\frac{(V'(x^+_h))^2 - 2V(x^+_h)V''(x^+_h) }{(V'(x^+_h))^3} > \frac{(V'(x^-_h))^2 - 2V(x^-_h)V''(x^-_h) }{(V'(x^-_h))^3},
\end{equation}
where $x^+_h>x_h^-$ are the $x$-intercepts of the orbit $H = h$. A similar condition, with the reversed inequality sign, implies $T$ is decreasing. Since normally, and certainly in our case, it is difficult to solve for $x^\pm_h$, we simplify the above condition by comparing all the values to the left and to the right with the value at the center. Thus, if for all $x>x_s$,
\begin{equation}\label{e:Fright}
\frac{(V'(x))^2 - 2V(x)V''(x) }{(V'(x))^3} > \lim_{x'\ra x_s}\frac{(V'(x'))^2 - 2V(x')V''(x') }{(V'(x'))^3} ,
\end{equation}
and for all $x <x_s$,
\begin{equation}\label{e:Fleft}
\frac{(V'(x))^2 - 2V(x)V''(x) }{(V'(x))^3} < \lim_{x'\ra x_s}\frac{(V'(x'))^2 - 2V(x')V''(x') }{(V'(x'))^3} ,
\end{equation}
then the period-function is increasing. Reversing the inequalities above gives a criterion for decreasing periods. 

To make subsequent computation easier let us scale the equilibrium of the system to $(1,0)$. We thus pass to a new couple of variables
\[
x \ra x/x_s, \quad y \ra  y/ x_s.
\]
The new system reads as follows
\begin{equation}\label{sys-xi-new}
\left\{\begin{split}
		x' & = y\\
		y' & = \lambda^2( -x + x^{\frac{\lambda-2}\lambda})
\end{split}\right. \ or\ 
\left\{\begin{split}
		x' & = -y\\
		y' & = \lambda^2( x - x^{\frac{\lambda-2}\lambda})
\end{split}\right.
\end{equation}
with the Hamiltonian given by
\begin{equation}\label{}
\begin{split}
H(x,y) &= \frac12 y^2 + \frac{\l^2}{2} V(x), \\
V(x) & = x^2 - \frac{\l}{\l-1} x^{2- 2/\l} + \frac{1}{\l-1} .
\end{split}
\end{equation}
In the new setup we are looking for monotonicity in the range $x \in \left( 0, \left( \frac{\l}{\l-1} \right)^{\l/2} \right)$, with $x_s = 1$ being the center. Computing the expressions involved in \eqref{e:Fright}, \eqref{e:Fleft}  we obtain
\[
\frac{(V'(x))^2 - 2V(x)V''(x) }{(V'(x))^3}  = 
\frac{-\frac{\l-2}{\l} x^{2-2/\l}+ x^{2-4/\l}-1 + \frac{\l-2}{\l} x^{-2/\l}}{2(\l-1)x^3(1-x^{-2/\l})^3},
\]
while
\[
\lim_{x\ra 1} \frac{(V'(x))^2 - 2V(x)V''(x) }{(V'(x))^3} = -\frac{(\l-2)}{12}.
\]

\subsubsection{Monotonicity of the period: subcase $\l>2$} We aim at showing that in this range the period function is increasing. To this end, we will verify \eqref{e:Fright}, \eqref{e:Fleft} for $V$ (clearly, a positive multiple of $V$ does not change the relations \eqref{e:Fright}, \eqref{e:Fleft} ). From the formulas above, \eqref{e:Fright} and \eqref{e:Fleft} are equivalent to one statement, namely, 
\begin{equation}\label{W}
\begin{split}
W(x) &= -\frac{\l-2}{\l} x^{2-2/\l} + x^{2-4/\l} -1 + \frac{\l-2}{\l} x^{-2/\l} \\
&+ \frac{(\l-1)(\l-2)}{6} x^3 (1- x^{-2/\l})^3 \geq 0,
\end{split}
\end{equation}
throughout  the interval $x  \in \left( 0, \left( \frac{\l}{\l-1} \right)^{\l/2} \right)$. 

In the range  $x >1$, to show \eqref{W} we notice that $W(1) = W'(1) = 0$, and 
\begin{multline*}
\frac{\l}{\l-2} \left( x^{1+2/\l} W'(x)  \right)' = \frac{4(\l-1)}{\l} x(x^{1-2/\l}-1)(1-x^{-2/\l}) +\\
+ \frac{(3\l+2)(\l-1)}{2} x^{2+2/\l} (1 - x^{-2/\l})^3+ 6(\l-1) x^2(1-x^{-2/\l})^2,
\end{multline*}
which is positive for $x >1$. Thus, $ x^{1+2/\l} W'(x)$ is increasing from $0$ to the right of $x=1$. Consequently, $W'(x)$ is positive for $x > 1$, and hence, so is $W$. 

For the range  $x \in (0,1)$, let us change
variables to $\z = x^{-2/\l}$. We have
\[
W(\z) = -\frac{\l-2}{\l} \z^{-\l+1} + \z^{-\l+2}-1 + \frac{\l-2}{\l} \z + \frac{(\l-1)(\l-2)}{6} \z^{-3\l/2} (1-\z)^3
\]
considered in the range $\z\geq 1$. By direct computation, we observe that 
both $\z^{3\l/2+2} W''(\z)$ and $(\z^{3\l/2+2} W''(\z))'$ vanish at $\z=1$. Thus, if we can show that $(\z^{3\l/2+2} W''(\z))'' \geq 0$ for all $\z\geq 1$, then it would imply that $\z^{3\l/2+2} W''(\z) \geq 0$, and in particular, $W'' \geq 0$, which gives $W \geq 0$ as before. 

So, computing the second derivative we have
\begin{multline*}
\frac{4}{(\l-1)(\l-2)} (\z^{3\l/2+2} W''(\z))'' = (\l+2)(\l+4) \z^{\l/2-1} (\z-1) \\
+ 4(\l+2) \z^{\l/2-1}+ 3(3\l -4)(\l-2) (1-\z) + 4(3\l-4):= F(\z).
\end{multline*}
We have 
\begin{equation}\label{}
\begin{split}
F'(\z) & = \frac{\l(\l+2)(\l+4)}{2} \z^{\l/2-1} - \frac{\l(\l+2)(\l-2)}{2}\z^{\l/2-2} - 3(3\l-4)(\l-2) \\
F''(\z) & = \frac{\l(\l+2)(\l-2)}{4}\z^{\l/2-3} \left[ (\l+4) \z - (\l-4) \right]\\
F'''(\z) & = \frac{\l(\l+2)(\l-2)(\l-4)}{8}\z^{\l/2-4} \left[ (\l+4) \z - (\l-6) \right].
\end{split}
\end{equation}
We see that $F(1) = 8(2\l-1) >0$, $F'(1) = -6(\l^2-6\l+4)$, and both $F''(\z), F'''(\z) > 0$ for all $\z\geq 1$. If $2< \l \leq 3+\sqrt{5}$, then $F'(1) \geq 0$, and thus, this case is settled by the Taylor expansion. If $\l > 3+\sqrt{5}$, then for $t = \z-1$ again by the Taylor expansion, we have the bound
\[
F(t) \geq 16\l-8 -6(\l^2-6\l+4) t + \l(\l^2-4) t^2.
\]
The discriminant of the quadratic on the right hand side is given by (up to a positive multiple)
\[
D = - 7\l^4 - 100\l^3+460\l^2-464\l+144,
\]
which is strictly negative for all $\l > 3+\sqrt{5}$. This settles the whole range of $\l >2$.

\subsubsection{Monotonicity of the period: subcase $\frac43 \leq \l <2$} We show that in this case the periods are monotonically decreasing. This amounts to proving the opposite inequality $W(x) \leq 0$ on $x  \in \left( 0, \left( \frac{\l}{\l-1} \right)^{\l/2} \right)$. For $x >1$, the term $(1-x^{-2/\l})$ is positive, and $x^2>x$. So,
\[
\frac{\l}{\l-2} \left( x^{1+2/\l} W'(x)  \right)' \geq x^2 (1-x^{-2/\l})^2 \left[ 6-\frac{4}{\l} + \frac{3\l+2}{2} x^{2/\l}(1-x^{-2/\l}) \right] \geq 0.
\]
This implies that $\left( x^{1+2/\l} W'(x)  \right)' \leq 0$, and hence $W\leq 0$.

In the case $x \leq 1$ we appeal to the previous computation. It is straightforward to see that each term in $F(\z)$ is non-negative as long as $\l \geq 4/3$. Hence $W''(\z) \leq 0$, and hence $W\leq 0$.

\section{Classification summary and further discussion}\label{s:summary}

Now we have all the tools to classify homogeneous solutions. We will only focus on the non-trivial case of $\l \neq 1$ as in this case all solutions are parallel shear flows focus. We derive the conjugate case of $0<\l<1$ by using transformation formula \eqref{e:conj}. Let us start with the elliptic case.

\subsection{Elliptic solutions}
In the range $1<\l$, note that monotonicity character of $T$ changes in the elliptic region depending on whether $\l>2$ or $\l<2$. At $\l = 2$ the solutions are described explicitly by \eqref{ex:2}. All the periods $T$ are $\pi$, yielding a non-trivial homogeneous solution for each $0<P<P_{\max}$ and of course the trivial pure rotational flow for $P = P_{\max}$. For other $\l$'s the situation is quite different. In order for $\psi$ to yield a field in $\R^2$, it has to be $2\pi$-periodic, which means that $T = \frac{2\pi}{n}$ for some $n\in \N$. In  view of \prop{p:mon} (i) there are no such solutions for any $\frac43<\l<2$ (except for the trivial rotational $P = P_{\max}$) since there is no integer $n$ satisfying $\pi< \frac{2\pi}{n} < \frac{2\pi}{\sqrt{2\l}}$. For $\l>2$, the number of $2\pi$-periodic solutions is equal to the number of integers $n\in \N$ so that $\frac{2\pi}{\sqrt{2\l}} < \frac{2\pi}{n} < \pi$. This means $n$ has to satisfy $4<n^2<2\l$. The first $n$ to satisfy $n^2>4$ is $3$, and so starting only after $\l > \frac{9}{2}$ the periodic solutions start to emerge. The number is given by the cardinality of the set $(2,\sqrt{2\l}) \cap \N$, which grows roughly like $\sqrt{2\l}$. 

In the range $0< \l < \frac34$ we argue by conjugacy. In this case we can rescale $B$ to $-1$ as the elliptic case corresponds to $B<0$ only. Then the range for the pressure is $P_{\min} = - \frac{1}{2\l} \left( \frac{1-\l}{\l^3} \right)^{\l-1} \leq P <0$, with $P = P_{\min}$ corresponding to the trivial rotation. Since $T_\l = \frac{1}{\l} T_{1/\l}$ under \eqref{e:conj} we argue that $\frac{\pi}{\l} <T_\l< \frac{2\pi}{\sqrt{2\l}}$ for $\frac12 < \l \leq \frac34$; $T_\l \equiv 2\pi$ for $\l = 
\frac12$ and all solutions are given explicitly by \eqref{ex:half}; and  $ \frac{2\pi}{\sqrt{2\l}} <T_\l<\frac{\pi}{\l}$ for $0 < \l < \frac12$. We can readily see that there is no fit for a $T$ of the form $ \frac{2\pi}{n}$ in either of the non-trivial cases. We summarize our findings in the table below.

\begin{table}[!h]
\centering
\begin{tabular}{|p{.6cm}<{\centering}|p{2.5cm}<{\centering}|p{2cm}<{\centering}|p{2.75cm}<{\centering}|p{2cm}<{\centering}|p{3cm}<{\centering}|}
\hline	
	\multirow{2}{*}{$B$}&\multirow{2}{*}{$P$}&\multicolumn{4}{c|}{$\lambda$}\\
\cline{3-6}
&&$(1,\frac43)$&$[\frac43,2)\cup(2,\frac92]$&$2$&$(\frac92,\infty)$\\
\hline
\multirow{2}{*}{1}&$P=P_{\max}$&\multicolumn{4}{c|}{rot.${^\dag}$}\\
\cline{2-6}
&$0<P<P_{\max}$&?&no$^{*}$&all $^{\diamond}$&$\#\{(2,\sqrt{2\lambda})\cap\mathbb{N}\}$\\
\hline
\end{tabular}
\begin{tabular}{|p{.6cm}<{\centering}|p{2.5cm}<{\centering}|p{3.32cm}<{\centering}|p{3.4cm}<{\centering}|p{3.4cm}<{\centering}|}
\hline
&&$(\frac34,1)$&$\frac12$&$(0,\frac12)\cup(\frac12,\frac34]$\\
\hline
\multirow{2}{*}{-1}&$P=P_{\min}$&\multicolumn{3}{c|}{rot.}\\
\cline{2-5}
&$P_{\min}<P	<0$&?&all&no\\
\hline
		\end{tabular}
\begin{flushleft}
	\footnotesize{${^\dag}$ ``rot." is short for the rotational flow described by example \eqref{ex:rot}} \\
	\footnotesize{$^\diamond$} ``all" means that all solutions are $2\pi$ periodic\\
	\footnotesize{$^*$ ``no" means there are no solutions}
\end{flushleft}
\bigskip
	\caption{\small{The number of elliptic periodic solution corresponding to all possible values of $B$, $P$ and $\lambda$ (after rescaling). Here, $P_{\max} =  \frac{1}{2\l} \left( \frac{\l-1}{\l^3} \right)^{\l-1}$, $P_{\min} = - \frac{1}{2\l} \left( \frac{1-\l}{\l^3} \right)^{\l-1}$}.}\label{table-elliptic}
\end{table}

\subsection{Hyperbolic and parabolic solutions}
Let us first discuss proper hyperbolic solutions. Since hyperbolic solutions may consist of different local solutions with the same pressure it makes sense to scale the pressure to a fixed value $P = \pm 1$ and indicate which life-spans are available for each sign of $B$. We thus list cases of $B<0$, $B>0$ and $B = 0$ based on \prop{p:mon} and conjugacy relation \eqref{e:conj}. According to \lem{l:ellex} and example \eqref{ex:half} the range $0<\l \leq \frac12$ enjoys no hyperbolic solutions. For $\l >\frac12$ any local solution is $H^1$, and stitching them together produces a globally $H^1$-solutions according to \lem{l:structure}. Let us focus on the case $\frac12 <\l <1$, as $\l>1$ is elaborated already in \prop{p:mon}. For $P = 1$, by conjugacy, $T_\l = \frac{1}{\l} T_{1/\l}$ as a function of $B >0$ varies between $0$ and $\pi$. For $P = -1$, this corresponds to $\tilde{B}>0$ which puts $T$ in the range $\pi < T< \frac{\pi}{\l}$. This makes it impossible to fit two local solutions with negative $P$ on the period of $2 \pi$. Thus no hyperbolic solutions exist in that range. In all other cases stitching is possible to produce a variety of solutions, and all solutions come in that form according to \lem{l:structure}. We summarize our findings in the table below.

\begin{table}[h]
	\begin{minipage}{.4\linewidth}
	\centering
        \begin{tabular}{|c|c|}
\hline
		$B$&$\lambda>1$\\
\hline
		$B>0$&$\frac\pi\lambda<T<\pi$\\
\hline
		$B=0$&$T=\frac\pi\lambda$\\
\hline
		$B<0$&$0<T<\frac\pi\lambda$\\
\hline
	\end{tabular}
	\vspace{2pt}
\subcaption{$P=-1$}
	\end{minipage}	
      \begin{minipage}{.4\linewidth}
	\centering
	  \begin{tabular}{|c|c|c|}
\hline
	$P$&$\frac12<\lambda<1$&$0<\lambda\leq\frac12$\\
\hline
	$P=1$&$0<T<\lambda$&no\\
\hline
	$P=-1$&no&no\\
\hline
        \end{tabular}
        \vspace{10pt}
        \subcaption{$B>0$}
	\end{minipage}
   	\caption{The ranges of the life-span $T$ as a function of $B$ for hyperbolic solutions.}\label{table-hyperbolic}
\end{table}

Let us recall that for $\l>1$ the hyperbolic solutions have the same slopes up to a sign at the points where they vanish. Thus, a $C^1$-smooth stitching is possible in this case. Stitching with slopes of opposite signs produces a vortex sheet along the ray where the stitching occurred ($\bu$ has a jump discontinuity in the tangential direction to the ray). In the range $\frac12 < \l <1$ the slopes are infinite at the end-points as seen from the proof of \lem{l:ellex}.

As to proper parabolic solutions, ones that vanish at one point only, we already excluded them in the range $\l>1$. In the range $0<\l< \frac12$ $H^1$-solutions aren't allowed to vanish by \lem{l:ellex}, and in the remaining range $\frac12 \leq \l<1$ we see the only unaccounted case is that of $B=0$ which yields a solution with period $\frac{\pi}{\l}$ giving $2\pi$ only when $\l = \frac12$. So, the parabolic solution we exhibited in \eqref{ex:half} is the only one that exist.

\subsection{Remark about the Onsager conjecture} One of our original motivations to study homogeneous solutions is to consider the particular case of $q = - \frac13$, or $\l = \frac23$, which yields a vector field $\bu$ locally in the Besov class $B^{1/3}_{3,\infty}$, i.e. Holder continuous in the $L^3$-sense. This space is known to be critical for a weak solution to the Euler equation to conserve energy in the sense that any smoother class would conserve energy and where are fields in that class that have anomalous energy class (see \cite{ds,isett,shv-lectures} for many references and detailed introduction in the problem). The ``hard" part of the Onsager conjecture is to show that there is an actual solution $\bu \in L^3_t (B^{1/3}_{3,\infty})_x$ that does not conserve energy. For the stationary case, such as ours, we aim at finding solutions with smooth or zero force for which the energy flux $\Pi$ is non-vanishing. By the latter we understand the value  resulting from testing the nonlinear term with a mollified field $\bu_\e$ and sending $\e \ra 0$ (this would certainly vanish for smooth fields). In our case, in view of \lem{l:ODEdist-q}, the force is zero, so \eqref{pq} is satisfied in $\R^2$. As shown in \cite{shv-lectures} the flux in this case is given by $\Pi = \int_0^{2\pi} (\psi'(\th))^3 d\th$ and is shown to vanish for any solution, by a direct manipulation with the equation \eqref{ODE}. Now that we can classify completely solutions for $\l = \frac23$ the reason for that becomes much more transparent. Since all solutions are either hyperbolic or constant, the local hyperbolic pieces are even with respect to their centers. So, $\psi'$ is odd on every local piece, and hence the integral for $\Pi$ is clearly zero. Recall that the underlying reason for the evenness of $\psi$ is the Hamiltonian structure of the ODE, inherited from the Hamiltonian structure of the original Euler equation. This factor has never been taken into account in previous studies of Onsager-critical solutions. For instance, the classical vortex sheets are critical too, but the energy conservation traces down to the incompressibility of $\bu$ (particles are not allowed to go across the sheet, see \cite{shv-lectures}). It would be interesting to get a deeper understanding of the Hamiltonian symmetries of the Euler equation in relation to the Onsager conjecture.


\end{document}